\documentclass[runningheads]{llncs}
\usepackage[T1]{fontenc}
\usepackage[utf8]{inputenc}
\usepackage{verbatim}

\usepackage{graphicx}

\usepackage{amsmath}
\usepackage{amssymb}
\usepackage{mathtools}
\newcommand{\R}{\mathbb{R}}
\newcommand{\IP}{\mathcal{P}}

\newcommand{\mat}[1]{\mathbf{#1}}

\usepackage{color}
\usepackage{soulutf8}
\soulregister\cite7
\soulregister\ref7
\soulregister\pageref7
\soulregister\eqref7

\newcommand{\tv}[1]{\hl{TV: #1}}

\newcommand{\para}[1]{\vspace{.5em}\noindent\textbf{#1}~}

\begin{document}

% ==============================================================================
%
\title{
    Functional Liftings of Vectorial Variational Problems with Laplacian
    Regularization
    %\thanks{Supported by organization x.}
}
\titlerunning{Functional Liftings of Vectorial Problems with Laplacian Regularization}
\author{
    % Anonymous
    Thomas Vogt%\orcidID{0000-0002-2135-4436}
    \and
    Jan Lellmann%\orcidID{0000-0002-5243-0331}
}
\authorrunning{
    % Anonymous
    T. Vogt and J. Lellmann
}
\institute{
    % Anonymous Institute 
    University of Lübeck,
    Institute of Mathematics and Image Computing (MIC),
    Maria-Goeppert-Str. 3, 23562 Lübeck
    \email{\{vogt,lellmann\}@mic.uni-luebeck.de}
}
\maketitle              % typeset the header of the contribution
\begin{abstract}
We propose a functional lifting-based convex relaxation of variational problems
with Laplacian-based second-order regularization.
The approach rests on ideas from the calibration method as well
as from sublabel-accurate continuous multilabeling approaches, and
makes these approaches amenable for variational problems with vectorial data
and higher-order regularization, as is common in image processing
applications.
We motivate the approach in the function space setting and prove that, in the
special case of absolute Laplacian regularization, it encompasses the
discretization-first sublabel-accurate continuous multilabeling approach as a special case.
We present a mathematical connection between the lifted and original functional
and discuss possible interpretations of minimizers in the lifted function
space.
Finally, we exemplarily apply the proposed approach to 2D image registration problems.
% The abstract should briefly summarize the contents of the paper in

\keywords{
         variational methods
    \and curvature regularization
    \and convex relaxation
    \and functional lifting
    \and measure-based regularization.}
\end{abstract}

% ==============================================================================
%

\section{Introduction}
Let $\Omega \subset \R^d$ and $\Gamma \subset \R^s$ both be bounded sets.
In the following, we consider the variational problem of minimizing the functional
\begin{equation}
    F(u) = \int_\Omega f(x,u(x),\Delta u(x)) dx,\label{eq:lapl-intro}
\end{equation}
%\jl{dx is not correct when using TV?}
%\tv{the highest-order term is the absolutely continuous part of the measure}
that acts on vector-valued functions $u \in C^2(\Omega;\Gamma)$.
Convexity of the integrand $f\colon \Omega \times \Gamma \times \R^s \to \R$ is
only assumed in the last entry, so that $u \mapsto F(u)$ is generally
\emph{non-convex}. The Laplacian $\Delta u$ is understood component-wise and reduces to $u''$ if the domain $\Omega$ is one-dimensional.

Variational problems of this form occur in a wide variety of image processing tasks, including image reconstruction, restoration, and interpolation.
Commonly, the integrand is split into data term and regularizer:
\begin{equation}\label{eq:f-decomp}
    f(x,z,p) = \rho(x,z) + \eta(p).
\end{equation}
As an example, in \emph{image registration} (sometimes referred to as large-displace\-ment optical flow), the data term $\rho(x,z) = d(R(x),T(x+z))$ encodes the pointwise distance of a reference
image $R\colon \R^d \to \R^k$ to a deformed template image
$T\colon \R^d \to \R^k$ according to a given distance measure $d(\cdot,\cdot)$,
such as the squared Euclidean distance $d(a,b) = \frac{1}{2}\|a-b\|_2^2$.
While often a suitable convex regularizer $\eta$ can be found, the highly non-convex nature of $\rho$ renders the search for global minimizers of
\eqref{eq:lapl-intro} a difficult problem.

Instead of directly minimizing $F$ using gradient descent or other local
solvers, we will aim to \emph{replace} it by a convex functional $\mathcal{F}$
that acts on a higher-dimensional (\emph{lifted}) function space.
If the lifting is chosen in such a way that we can construct global minimizers
of $F$ from global minimizers of $\mathcal{F}$, we can find a global solution
of the original problem by applying convex solvers to~$\mathcal{F}$. While  we cannot claim this property for our choice of lifting, we believe that the mathematical motivation and some of the experimental
results show that this approach can be a good basis for future work on global solutions of variational models with higher-order regularization.

%the proposed lifted functional might in general have global minimizers
%that are difficult to relate to global minimizers of the original problem
%without prior knowledge.
%We still think that the mathematical motivation and some of the experimental
%results show reasonably well that this approach can be valuable for future work on global solutions of variational models with higher order regularization.

\para{Calibrations in variational calculus}
The lifted functional $\mathcal{F}$ proposed in this work is motivated by
previous lifting approaches for \emph{first-order} variational problems of the form %that exhibit the desired property, 

%Many image processing problems can be formulated as variational problems
\begin{equation}\label{eq:first-order-functional}
    \min_u F(u) = \int_\Omega f(x,u(x),\nabla u(x)) dx,
\end{equation}
where $F$ acts on functions $u\colon \Omega \to \Gamma$ with
$\Omega \subset \R^d$ and scalar range $\Gamma \subset \R$.

The \emph{calibration method} as introduced in \cite{alberti2003} gives a globally
sufficient optimality condition for functionals of the
form \eqref{eq:first-order-functional} with $\Gamma = \R$.
Importantly, $f(x,z,p)$ is not required to be convex in $(x,z)$, but only in $p$.
The method states that $u$ minimizes $F$ if there exists a divergence-free
vector field $\phi\colon \Omega \times \R \to \R^{d+1}$ (a \emph{calibration}) in a certain
admissible set $X$ of vector fields on $\Omega \times \R$ (see below for details), such that
\begin{equation}
    F(u) = \int_{\Omega \times \R} \phi \cdot D\mat{1}_u,
\end{equation}
where $\mat{1}_u$ is the characteristic function of the subgraph of $u$ in
$\Omega \times \R$, $\mat{1}_u(x,z)=1$ if $u(x) > z$ and $0$ otherwise, and $D \mat{1}_u$ is its distributional derivative.
%\begin{equation}
%    \mat{1}_u(x,z) = \begin{cases}
%        1 & \text{if } u(x) > z, \\
%        0 & \text{else.}
%    \end{cases}
%\end{equation}
The duality between subgraphs and certain vector fields is also the subject of
the broader theory of Cartesian currents \cite{giaquinta1998}.

A convex relaxation of the
original minimization problem then can be formulated in a higher-dimensional space
by considering the functional \cite{chambolle2001,pock2010}
\begin{equation}
    \mathcal{F}(v) := \sup_{\phi \in X}
        \int_{\Omega \times \R} \phi \cdot Dv,
\end{equation}
acting on functions $v$ from the \emph{convex} set
\begin{equation}
    \mathcal{C} = \{ v\colon \Omega \times \R \to [0,1] :
        \lim_{z \to -\infty} v(x,z) = 1,
        \lim_{z \to \infty} v(x,z) = 0
    \}.
\end{equation}
In both formulations, the set of admissible test functions is 
\begin{align}
    X = \{ \phi\colon \Omega \times \R \to \R^{d+1} : ~
        &\phi^t(x,z) \geq f^*(x,z,\phi^x(x,z)) \\
        &\text{for every } (x,z) \in \Omega \times \R
    \},
\end{align}
where $f^*(x,z,p) := \sup_q \langle p, q \rangle - f(x,z,q)$ is the convex
conjugate of $f$ with respect to the last variable.
In fact, the equality
\begin{equation}
    F(u) = \mathcal{F}(\mat{1}_u)
\end{equation}
has been argued to hold for $u \in W^{1,1}(\Omega)$ under suitable assumptions
on $f$ \cite{pock2010}.
A rigorous proof of the case of $u \in BV(\Omega)$ and $f(x,z,p) = f(z,p)$
($f$ independent of $x$), but not necessarily continuous in $z$, can be
found in the recent work \cite{bouchitte2018}.

In \cite{mollenhoff2017}, it is discussed how the choice of discretization
influences the results of numerical implementations of this approach.
More precisely, motivated by the work \cite{mollenhoff2016} from continuous
multilabeling techniques, the choice of piecewise linear finite elements on
$\Gamma$ was shown to exhibit so-called \emph{sublabel-accuracy,} which
is known to significantly reduce memory requirements.

\para{Vectorial data}
The application of the calibration method to \emph{vectorial} data $\Gamma \subset \R^s$,
$s > 1$, is not straightforward, as the concept of subgraphs, which is central
to the idea, does not translate easily to higher-dimensional range.
While the original sufficient minimization criterion has been successfully
translated \cite{mora2002}, functional lifting approaches have not been based on
this generalization so far.
In \cite{strekalovskiy2012}, this approach is considered to be intractable in terms of
memory and computational performance.

There are functional lifting approaches for vectorial data with first-order
regularization that consider the subgraphs of the components of $u$
\cite{goldluecke2013,strekalovskiy2014}.
It is not clear how to generalize this approach to nonlinear data
$\Gamma \subset \mathcal{M}$, such as a manifold $\mathcal{M}$, where other
functional lifting approaches exist at least for the case of total variation
regularization \cite{lellmann2013}.

An approach along the lines of \cite{mollenhoff2016} for vectorial data with
total variation regularization was proposed in \cite{laude2016}.
Even though \cite{mollenhoff2017} demonstrated how \cite{mollenhoff2016} can be
interpreted as a discretized version of the calibration-based lifting, the
equivalent approach \cite{laude2016} for vectorial data lacks a fully-continuous
formulation as well as a generalization to arbitrary integrands that would
demonstrate the exact connection to the calibration method.

\para{Higher-order regularization}
Another limitation of the calibration method is its limitation to first-order derivatives of $u$, which leaves out
higher-order regularizers such as the Laplacian-based curvature regularizer in image registration \cite{fischer2003}.
Recently, a functional lifting approach has been successfully applied to
second-order regularized image registration problems \cite{loewenhauser2018},
but the approach was limited to a single regularizer, namely the integral
over the $1$-norm of the Laplacian (\emph{absolute Laplacian regularization}).

\para{Projection of lifted solutions} In the scalar-valued case with first-order
regularization, the calibration-based lifting is known to generate minimizers
that can be \emph{projected} to
minimizers of the original problem by thresholding \cite[Theorem 3.1]{pock2010}.
This method is also used for vectorial data with component-wise lifting as
in \cite{strekalovskiy2014}.
In the continuous multi-labeling approaches
\cite{lellmann2013,mollenhoff2016,laude2016}, simple averaging is demonstrated
to produce useful results even though no theoretical proof is given addressing
the accuracy in general.
In convex LP relaxation methods, projection (or \emph{rounding}) strategies with provable optimality bounds exist \cite{kleinberg2002} and can be extended to the continuous setting~\cite{Lellmann2012}.
We demonstrate that rounding is non-trivial in our case, but will leave a thorough investigation to future work.

\para{Contribution} In Section~\ref{sec:theory}, we propose a calibration method-like functional
lifting approach in the fully-continuous vector-valued setting for functionals
that depend in a convex way on $\Delta u$.
We show that the lifted functional satisfies $\mathcal{F}(\delta_u) \leq F(u)$,
where $\delta_u$ is the lifted version of a function $u$ and discuss the question
of whether the inequality is actually an equality.
For the case of absolute Laplacian regularization, we show that our model is
a generalization of \cite{loewenhauser2018}.
In Section~\ref{sec:theory-numerics}, we clarify how convex saddle-point solvers can be applied to our discretized model.
Section~\ref{sec:results} is concerned with experimental results.
We discuss the problem of projection and demonstrate that the model
can be applied to image registration problems.

\section{A calibration method with vectorial second-order terms}%
\label{sec:theory}

\subsection{Continuous formulation}%
\label{sec:theory-continuous}

We propose the following lifted substitute for $F$:
\begin{equation}\label{eq:F-lifted}
    \mathcal{F}(\mat{u}) := \sup_{(p,q) \in X}
        \int_{\Omega}\int_{\Gamma}
        (\Delta_x p(x,z) + q(x,z)) \,d\mat{u}_x(z) dx,
\end{equation}
acting on functions $\mat{u}\colon \Omega \to \IP(\Gamma)$ with values in the
space $\IP(\Gamma)$ of Borel probability measures on $\Gamma$.
This means that, for each $x \in \Omega$ and any measurable set
$U \subset \Gamma$, the expression $\mat{u}_x(U) \in \R$ can be interpreted as
the ``confidence'' of an assumed underlying function on $\Omega$ to take a value
inside of $U$ at point~$x$.
A function $u\colon \Omega \to \Gamma$ can be \emph{lifted} to a function
$\mat{u}\colon \Omega \to \IP(\Gamma)$ by defining $\mat{u}_x := \delta_{u(x)}$,
the Dirac mass at $u(x) \in \Gamma$, for each $x \in \Omega$.
%\jl{Dirac Definition weglassen wenn zu lang}
%That is, for any measurable $U \subset \Gamma$, we have
%\begin{equation}
%    \delta_{z}(U) := \begin{cases}
%        1 & \text{ if } z \in U, \\
%        0 & \text{ otherwise.}
%    \end{cases}
%\end{equation}

We propose the following set of test functions in the definition of $\mathcal{F}$:
\begin{align}
    X = \{ (p,q): ~&
        p \in C_c^2(\Omega \times \Gamma), q \in L^1(\Omega \times \Gamma),\\
        \label{eq:concavity}
        &z \mapsto p(x,z) \text{ concave } \\
        \label{eq:fstar-ineq}
        &\text{and } q(x,z) + f^*(x,z,\nabla_z p(x,z)) \leq 0 \\
        &\text{for every } (x,z) \in \Omega \times \Gamma
    \},
\end{align}
where $f^*(x,z,q) := \sup_{p \in \R^s} \langle q,p \rangle - f(x,z,p)$ is the
convex conjugate of $f$ with respect to the last argument.

A thorough analysis of $\mathcal{F}$ requires a careful choice of function
spaces in the definition of $X$ as well as a precise definition of the
properties of the integrand $f$ and the admissible functions
$\mat{u}\colon \Omega \to \IP(\Gamma)$, which we leave to future work. Here, we present a proof that the lifted functional $\mathcal{F}$ bounds the original functional $F$ from below.

\begin{proposition}
    Let $f\colon \Omega \times \Gamma \times \R^s \to \R$ be measurable in the
    first two, and convex in the third entry, and let $u \in C^2(\Omega;\Gamma)$
    be given.
    Then, for $\mat{u}\colon \Omega \to \IP(\Gamma)$ defined by 
    $\mat{u}_x := \delta_{u(x)}$, it holds that
    \begin{equation}\label{eq:F-Flift-inequality}
        F(u) \geq \mathcal{F}(\mat{u}).
    \end{equation}
\end{proposition}

\begin{proof}
    Let $p,q$ be any pair of functions satisfying the properties from the
    definition of $X$.
    By the chain rule, we compute 
    \begin{align}\label{eq:delta-p-chain-rule}
        \Delta_x p(x,u(x))
        &= \Delta\left[p(x,u(x))\right]
            - \sum_{i=1}^d 
                \langle
                    \partial_i u(x),
                    D^2_z p(x,u(x)) \partial_i u(x)
                \rangle \\
        &\phantom{=} - 2\langle \nabla_x \nabla_z p(x,u(x)), \nabla u(x) \rangle
            - \langle \nabla_z p(x,u(x)), \Delta u(x) \rangle.\nonumber
    \end{align}
    Furthermore, the divergence theorem ensures
    \begin{align}\label{eq:dxdz-gauss}
        -\int_\Omega \langle \nabla_x \nabla_z p(x,u(x)), \nabla u(x) \rangle dx
        &= \int_\Omega \langle \nabla_z p(x,u(x)), \Delta u(x) \rangle dx \\
        &\phantom{=} + \int_\Omega\sum_{i=1}^d
                \langle \partial_i u(x), D^2_z p(x,u(x)) \partial_i u(x) \rangle
        dx,\nonumber
    \end{align}
    as well as $\int_\Omega \Delta\left[p(x,u(x))\right] dx = 0$
%    \begin{equation}\label{eq:delta-p-gauss}
%        
%    \end{equation}
    by the compact support of $p$. As $p \in C_c^2(\Omega \times \Gamma)$, concavity of
    $z \mapsto p(x,z)$ implies a negative semi-definite Hessian
    $D^2_z p(x,z)$, so that, together with
    \eqref{eq:delta-p-chain-rule}--\eqref{eq:dxdz-gauss},
    \begin{align}\label{eq:lapl-est}
        \int_{\Omega} \Delta_x p(x,u(x)) \,dx
        \leq \int_{\Omega} \langle \nabla_z p(x,u(x)), \Delta u(x) \rangle \,dx.
    \end{align}   
We conclude %\jl{removed first eq (not good)}
    \begin{align}
        \mathcal{F}(\mat{u})
        &=\int_{\Omega}\int_{\Gamma}
            (\Delta_x p(x,z) + q(x,z)) \,d\mat{u}_x(z) dx \\
        &= \int_{\Omega} \Delta_x p(x,u(x)) + q(x,u(x)) \,dx \\
        &\overset{\eqref{eq:fstar-ineq}}{\leq} \int_{\Omega} \Delta_x p(x,u(x)) - f^*(x,u(x),\nabla_z p(x, u(x))) \,dx \\
        &\overset{\eqref{eq:lapl-est}}{\leq} \int_{\Omega}
            \langle \nabla_z p(x,u(x)), \Delta u(x) \rangle
            - f^*(x,u(x),\nabla_z p(x, u(x))) \,dx \\
        &\leq \int_{\Omega} f(x,u(x), \Delta u(x)) \,dx,
    \end{align}
    where we used the definition of $f^*$ in the last inequality.
    \qed
\end{proof}

By a standard result from convex analysis,
$
    \langle p,g \rangle - f^*(x,z,g) = f(x,z,p)
$
whenever $g \in \partial_p f(x,z,p)$, the subdifferential of $f$ with
respect to $p$.
Hence, for equality to hold in \eqref{eq:F-Flift-inequality}, we would need to find a
function $p \in C_c^2(\Omega \times \Gamma)$ with
\begin{equation}\label{eq:optimal-p}
    \nabla_z p(x,u(x)) \in \partial_p f(x,u(x), \Delta u(x))
\end{equation}
and associated $q(x,z) := -f^*(x,z,\Delta u(x))$, such that $(p,q) \in X$ or $(p,q)$ can be approximated by functions from $X$.
% for \jl{, or }or.
%Alternatively, it would suffice to approximate this pair $(p,q)$ with
%functions from the set $X$ in a suitable sense. \jl{We expect that a more rigg}
%However, the exact choice of function spaces in the definition of $X$ as well
%as the assumptions on $f$ have to be chosen accordingly which is beyond the 
%scope of this work.

\para{Separate data term and regularizer} %
If the integrand can be decomposed into $f(x,z,p) = \rho(x,z) + \eta(p)$ as in \eqref{eq:f-decomp}, with 
$\eta \in C^1(\R^s)$ and $u$ sufficiently smooth, the optimal pair
$(p,q)$ in the sense of \eqref{eq:optimal-p} can be explicitly given as
\begin{align}
    p(x,z) &:= \langle z, \nabla \eta(\Delta u(x)) \rangle, \\
    q(x,z) &:= \rho(x,z) - \eta^*(\nabla \eta(\Delta u(x))).
\end{align}
A rigorous argument that such $p,q$ exist for any given $u$ could be made by approximating them by compactly supported functions from the admissible set~$X$ using suitable cut-off functions on $\Omega \times \Gamma$.

\subsection{Connection to the discretization-first approach \cite{loewenhauser2018}}%
\label{sec:theory-discrete}

In \cite{loewenhauser2018}, data term $\rho$ and regularizer $\eta$ are lifted
independently from each other for the case $\eta = \| \cdot \|_1$.
Following the continuous multilabeling approaches in
\cite{chambolle2012,mollenhoff2016,laude2016}, the setting is fully discretized
in $\Omega \times \Gamma$ in a first step.
Then the lifted data term and regularizer are defined to be the convex hull
of a constraint function, which enforces the lifted terms to agree on the Dirac measures $\delta_u$ with the
original functional applied to the corresponding function $u$.
The data term is taken from \cite{laude2016}, while the main contribution concerns the regularizer that now depends on the
Laplacian of $u$.

In this section, we show that our fully-continuous lifting is a generalization
of the result from \cite{loewenhauser2018} after discretization.

\para{Discretization}
In order to formulate the discretization-first lifting approach given in
\cite{loewenhauser2018}, we have to clarify the used discretization.

For the image domain $\Omega \subset \R^d$, discretized using points
$X^1, \dots, X^N \in \Omega$ on a rectangular grid, we employ a finite-differences scheme:
We assume that, on each grid point $X^{i_0}$, the discrete Laplacian of
$u \in \R^{N,s}$, $u^{i} \approx u(X^i) \in \R^s$, is defined using the
values of $u$ on $m+1$ grid points $X^{i_0}, \dots, X^{i_m}$ such that
\begin{equation}\textstyle
    (\Delta u)^{i_0} = \sum_{l=1}^m (u^{i_l} - u^{i_0}) \in \R^s.
\end{equation}
For example, in the case $d = 2$, the popular five-point stencil means $m = 4$
and the $X^{i_l}$ are the neighboring points of $X^{i_0}$ in the rectangular
grid.
More precisely,
\begin{align}\textstyle
    \sum_{l=1}^4 (u^{i_l} - u^{i_0})
    &= [u^{i_1} - 2u^{i_0} + u^{i_2}] + [u^{i_3} - 2u^{i_0} + u^{i_4}].
\end{align}
The range $\Gamma \subset \R^s$ is triangulated into simplices
$\Delta_1,\dots,\Delta_M$ with altogether $L$ vertices (or \emph{labels})
$Z^1, \dots, Z^L \in \Gamma$.
We write
%\[
    $T := (Z^1|\dots|Z^L)^T \in \R^{L,s},$
%\]
and define the sparse indexing matrices $P^j \in \R^{s+1,L}$ in 
such a way that the rows of
$T_j := P^j T \in \R^{s+1,s}$ are the labels that make up $\Delta_j$.

There exist piecewise linear finite elements
$\Phi_k\colon \Gamma \to \R$, $k = 1,\dots,L$ satisfying $\Phi_k(t_l)=1$ if
$k=l$, and $\Phi_k(t_l)=0$ otherwise.
%\begin{equation}
%    \Phi_k(t_l) = \begin{cases}
%        1 & \text{if } k = l, \\
%        0 & \text{otherwise.}
%    \end{cases}
%\end{equation}
In particular, the $\Phi_k$ form a partition of unity for $\Gamma$, i.e., $\sum_k \Phi_k(z) = 1 \text{ for any } z \in \Gamma$.
%\begin{equation}
%    \sum_k \Phi_k(z) = 1 \text{ for any } z \in \Gamma.
%\end{equation}
For a function $p\colon \Gamma \to \R$ in the function space spanned by the
$\Phi_k$, with a slight abuse of notation, we write $p = (p_1,\dots,p_L)$, where
$p_k = p(Z^k)$ so that
$
    p(z) = \sum_k p_k \Phi_k(z).
$

\para{Functional lifting of the discretized absolute Laplacian}
Along the lines of classical continuous multilabeling approaches, the absolute
Laplacian regularizer is lifted to become the convex hull of the constraint function
$\phi: \R^L \to \R \cup \{+\infty\}$,
\begin{equation}\label{eq:discrete-phi-function}
    \phi(p) := \begin{cases}
        \mu \left\| \sum_{l=1}^m (T_{j_l} \alpha^l - T_{j_0} \alpha^0) \right\|,
            & \text{if } p = \mu \sum_{l=1}^m (P^{j_l} \alpha^l - P^{j_l} \alpha^0), \\
        +\infty, & \text{otherwise,}
    \end{cases}
\end{equation}
where $\mu \geq 0$, $\alpha^l \in \Delta^U_{s+1}$ (for $\Delta^U_{s+1}$ the unit
simplex) and $1 \leq j_l \leq M$ for each $l = 0,\dots,m$.
The parameter $\mu \geq 0$ is enforcing positive homogeneity of $\phi$ which
makes sure that the convex conjugate $\phi^*$ of $\phi$ is given by the
characteristic function $\delta_\mathcal{K}$ of a set
$\mathcal{K} \subset \R^L$.
Namely,
\begin{align}\label{eq:Kset}\textstyle
    \mathcal{K} = \bigcap_{1 \leq j_l \leq M} \{
        f \in \R^{L}:
        &\textstyle
        \sum_{l=1}^m (f(t^l) - f(t^0))
            \leq \left\| \sum_{l=1}^m (t^l - t^0) \right\|, \\
        &\textstyle
        \text{for any } \alpha^l \in \Delta^U_{s+1}, l=0,1,\dots,m
    \},
\end{align}
where $t^l := T_{j_l} \alpha^l$ and $f(t^l)$ is the evaluation of the piecewise
linear function $f$ defined by the coefficients $(f_1,\dots,f_L)$ (cf. above).
The formulation of $\mathcal{K}$ comes with infinitely many constraints so far.

We now show two propositions which give a meaning to this set of constraints for arbitrary dimensions $s$ of the labeling space and an arbitrary choice of norm in the definition of $\eta = \|\cdot\|$. They extend the component-wise (anisotropic) absolute Laplacian result in \cite{loewenhauser2018} to the vector-valued case.

\begin{proposition}\label{prop:K-concave-lipschitz}
    The set $\mathcal{K}$ can be written as
    \[
        \mathcal{K} = \left\{ f \in \R^{L}:
            f\colon \Gamma \to \R \text{ is concave and 1-Lipschitz continuous}
        \right\}.
    \]
\end{proposition}

\begin{proof}
    If the piecewise linear function induced by $f \in \R^L$ is concave and
    1-Lipschitz continuous, then
    \begin{align}
        \frac{1}{m} \sum_{l=1}^m (f(t^l) - f(t^0)) 
        &= \left( \frac{1}{m} \sum_{l=1}^m f(t^l) \right) - f(t^0)
        \leq f\left(\frac{1}{m} \sum_{l=1}^m t^l\right) - f(t^0) \\
        &\leq \left\| \left( \frac{1}{m} \sum_{l=1}^m t^l \right) - t^0 \right\|
        = \frac{1}{m} \left\| \sum_{l=1}^m (t^l - t^0) \right\|.
    \end{align}
    Hence, $f \in \mathcal{K}$.
    On the other hand, if $f \in \mathcal{K}$, then we recover Lipschitz
    continuity by choosing $t^l = t^1$, for any $l$ in \eqref{eq:Kset}.
    For concavity, we first prove mid-point concavity. That is, for any
    $t^1, t^2 \in \Gamma$, we have
    \begin{equation}\textstyle
        \frac{f(t^1) + f(t^2)}{2} \leq f\left(\frac{t^1 + t^2}{2}\right)
    \end{equation}
    or, equivalently, $[f(t^1) - f(t^0)] + [f(t^2) - f(t^0)] \leq 0$,
    where $t^0 = \frac{1}{2}(t^1 + t^2)$.
    This follows from \eqref{eq:Kset} by choosing $t^0 = \frac{1}{2}(t^1 + t^2)$
    and $t^l = t^0$ for $l > 2$.
    With this choice, the right-hand side of the inequality in \eqref{eq:Kset}
    vanishes and the left-hand side reduces to the desired statement.
    Now, $f$ is continuous by definition and, for these functions, mid-point
    concavity is equivalent to concavity.
    \qed
\end{proof}

The following theorem is an extension of \cite[Theorem 1]{loewenhauser2018} to the vector-valued case and is crucial for numerical performance, as it shows that the constraints in Prop. \ref{prop:K-concave-lipschitz} can be reduced to a finite number:
%stated using
%finitely many constraints.
%As pointed out in  for the cases $s=1$, the
%constraints in Prop. \ref{prop:K-concave-lipschitz} can be stated using
%finitely many constraints. This :

\begin{proposition}\label{prop:K-finite-constraints}
    The set $\mathcal{K}$ can be expressed using not more than $|\mathcal{E}|$
    (nonlinear) constraints, where $\mathcal{E}$ is the set of faces (or edges
    in the 2D-case) in the triangulation.
\end{proposition}

\begin{proof}
    Usually, Lipschitz continuity of a piecewise linear function requires one
    constraint on each of the simplices in the triangulation, and thus as many constraints as there are gradients. However, together with concavity, it suffices to enforce a gradient
    constraint on each of the boundary simplices, of which there are fewer than the number
    of outer faces in the triangulation.
    This can be seen by considering the one-dimensional case
    where Lipschitz constraints on the two outermost pieces of a concave
    function enforce Lipschitz continuity on the whole domain.
    Concavity of a function $f\colon \Gamma \to \R$ expressed
    in the basis $(\Phi_k)$ is equivalent to its gradient being monotonously
    decreasing across the common boundary between any neighboring simplices.
    Together, we need one gradient constraint for each inner, and at most one
    for each outer face in the triangulation.
    \qed
\end{proof}

\subsection{Numerical aspects}%
\label{sec:theory-numerics}

For the numerical experiments, we restrict to the special case of integrands
$f(x,z,p) = \rho(x,z) + \eta(p)$ as motivated in Section \ref{sec:theory-continuous}.

\para{Discretization} %\label{sec:discretization2}
We base our discretization on the setting in Section \ref{sec:theory-discrete}.
%and go on clarifying the following choices:
%\tv{discuss boundary conditions for the finite difference schemes?}
For a function $p\colon \Gamma \to \R$ in the function space spanned by the
$\Phi_k$, we note that
\begin{equation}\label{eq:p-bary-repr}\textstyle
    p(z) = \sum_{k=1}^L p_k \Phi_k(z) = \langle A^jz - b^j, P^jp \rangle 
    \text{ whenever } z \in \Delta_j,
\end{equation}
where $A^j$ and $b^j$ are such that $
    \alpha = A^jz - b^j \in \Delta^U_{s+1}
$ contains
the barycentric coordinates of $z$ with respect to $\Delta_j$.
More precisely, for $\bar{T}^j := (P^j T \vert -e)^{-1} \in \R^{s+1,s+1}$ with
$e = (1,\dots,1) \in \R^{s+1}$, we set
\begin{align}
    A^j := \bar{T}^j\texttt{(1:s,:)} \in \R^{s,s+1}, \quad
    b^j := \bar{T}^j\texttt{(s+1,:)} \in \R^{s+1}.
\end{align}
The functions $\mat{u}\colon \Omega \to \IP(\Gamma)$ are discretized as
$
    u^{ik} := \int_{\Gamma} \Phi_k(z) d\mat{u}_{X^i}(z)
$,
hence $u \in \R^{N,L}$.
Furthermore, whenever $\mat{u}_x = \delta_{u(x)}$, the  discretization
$u^{i}$ contains the barycentric coordinates of $u(X^i)$ relative to $\Delta_j$.
In the context of first-order models, this property is described as sublabel-accuracy in \cite{laude2016,mollenhoff2017}.

\para{Dual admissibility constraints} The admissible set $X$ of dual variables is realized by discretizing
the conditions \eqref{eq:concavity} and \eqref{eq:fstar-ineq}.

Concavity \eqref{eq:concavity} of a function $p\colon \Gamma \to \R$ expressed
in the basis $(\Phi_k)$ is equivalent to its gradient being monotonously
decreasing across the common boundary between any neighboring simplices. This amounts to 
%If the (piecewise constant) gradient of $p$ is given by
%\begin{equation}
%    \nabla p(z) := g^j \text{ whenever } z \in \Delta_j,
%\end{equation}
%concavity can be encoded as
\begin{equation}
    \langle g^{j_2} - g^{j_1}, n_{j_1,j_2} \rangle \leq 0,
\end{equation}
where $g^{j_1},g^{j_2}$ are the (piecewise constant) gradients $\nabla p(z)$ on two neighboring simplices $\Delta_{j_1},\Delta_{j_2}$, and $n_{j_1,j_2} \in \R^s$ is the normal of their common
boundary pointing from $\Delta_{j_1}$ to $\Delta_{j_2}$.

The inequality \eqref{eq:fstar-ineq} is discretized using \eqref{eq:p-bary-repr}
similar to the one-dimensional setting presented in \cite{mollenhoff2017}.
We denote the dependence of $p$ and $q$ on $X^i \in \Omega$ by a superscript
$i$ as in $q^i$ and $p^i$.
Then, for any $j = 1,\dots,M$, we require
\begin{equation}
    \sup_{z \in \Delta_j}
        \langle A^j z - b^j, P^jq^i \rangle  - \rho(X^i, z) + \eta^*(g^{ij}) \leq 0
\end{equation}
which, for $\rho_j := \rho + \delta_{\Delta_j}$, can be formulated equivalently as
\begin{equation}
   \rho_j^*(X^i,(A^j)^T P^jq^i) + \eta^*(g^{ij}) \leq \langle b^j, P^jq \rangle.
\end{equation}

%\para{Saddle point form for PDHG} Plugging everything together, 
The fully discretized problem can be expressed
in convex-concave saddle point form to which we apply the primal-dual hybrid
gradient (PDHG) algorithm \cite{chambolle2011} with adaptive step sizes
from \cite{goldstein2013}.
The epigraph projections for $\rho_j^*$ and $\eta$ are implemented along the
lines of \cite{mollenhoff2016} and \cite{pock2010}.

\begin{comment}
\begin{align*}
    \min_{u} \max_{p,q,g} \quad
        & \sum_{i}  \langle u^{i}, (\Delta p)^i + q^i \rangle \\
    \text{s.t.}\quad
        & \sum_k u^{ki} = 1, u^i \geq 0, g^{ij} = B^j P^j p^i, \\
        & \rho_j^*(X^i,(A^j)^T P^jq^i) + \eta^*(g^{ij})
            \leq \langle b^j, P^j q^i \rangle, \\
        & \langle g^{j_2} - g^{j_1}, n_{j_1,j_2} \rangle \leq 0.
\end{align*}
\tv{add formulation of primal objective function and address difficulty of evaluation?}
\end{comment}

\section{Numerical results}%
\label{sec:results}
We implemented the proposed model in Python 3 with NumPy and PyCUDA.
%the GPU programming framework PyCUDA.
The examples were computed on an Intel Core i7 4.00\,GHz with 16\,GB of memory
and an NVIDIA GeForce GTX 1080 Ti with 12\,GB of dedicated video memory.
The iteration was stopped when the Euclidean norms of the primal and
dual residuals \cite{goldstein2013} fell below
$10^{-6} \cdot \sqrt{n}$ where $n$ is the respective number of variables. % in each. % respective primal or dual variables.
%We confirmed that the solutions remained visually stable for more restrictive
%stopping criteria.

\begin{figure}[t]
    \includegraphics[
        trim=125 464 575 87,clip,width=0.22\textwidth
    ]{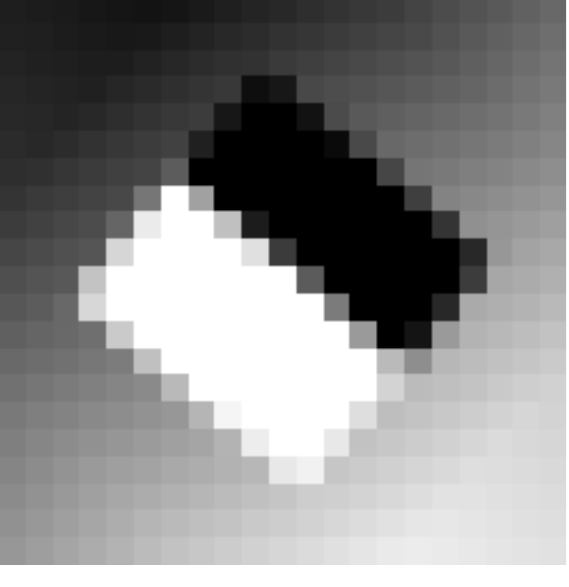}
    \hfill
    \includegraphics[
        trim=360 464 340 87,clip,width=0.22\textwidth
    ]{fig/plot-rotate.pdf}
    \hfill
    \includegraphics[
        trim=598 464 102 87,clip,width=0.22\textwidth
    ]{fig/plot-rotate.pdf}
    \hfill
    \includegraphics[
        trim=580 270 85 278,clip,width=0.272\textwidth
    ]{fig/plot-rotate.pdf}
    \hfill
    
    \vspace{-0.3em}
    \raisebox{0.13\height}{\includegraphics[
        trim=326 257 322 280,clip,width=0.24\textwidth
    ]{fig/plot-rotate.pdf}}
    \hfill
    \includegraphics[
        trim=605 80 115 478,clip,width=0.24\textwidth
    ]{fig/plot-rotate.pdf}%
    \hspace{-0.3em}\raisebox{0.13\textwidth}{$\mapsto\,$}%
    \includegraphics[
        trim=135 80 585 478,clip,width=0.24\textwidth
    ]{fig/plot-rotate.pdf}
    \hfill
    \includegraphics[
        trim=125 250 580 279,clip,width=0.20\textwidth
    ]{fig/plot-rotate.pdf}
    
    \vspace{-1.1em}
    \caption{
        Application of the proposed higher-oder lifting to image registration with
        SSD data term and squared Laplacian regularization.
        The method accurately finds a deformation (bottom row, middle and right)
        that maps the template image (top row, second from left) to the
        reference image (top row, left), as also visible from the difference
        image (top row, right).
        The result (top row, second from right) is almost pixel-accurate,
        although the range $\Gamma$ of possible deformation vectors at each
        point is discretized using only $25$ points (second row, left).
    }\label{fig:result-rotate}
    \vspace{-1.5em}
\end{figure}%
\begin{figure}
    \centering
    \hfill
    \includegraphics[
        trim=161 471 612 87,clip,width=0.145\textwidth
    ]{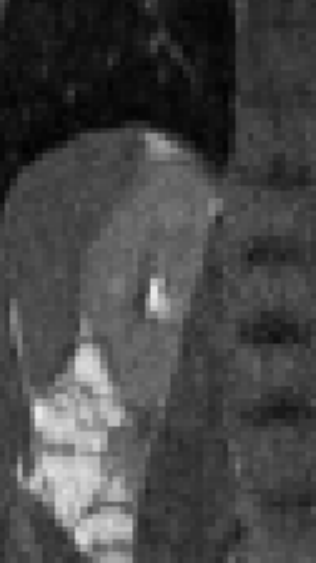}%
    \hfill
    \includegraphics[
        trim=397 471 376 87,clip,width=0.145\textwidth
    ]{fig/plot-dce-mri.pdf}%
    \hfill
    \includegraphics[
        trim=633 471 140 87,clip,width=0.145\textwidth
    ]{fig/plot-dce-mri.pdf}%
    \hfill
    \includegraphics[
        trim=166 86 617 471,clip,width=0.147\textwidth
    ]{fig/plot-dce-mri.pdf}%
    \hfill
    \includegraphics[
        trim=647 79 125 471,clip,width=0.145\textwidth
    ]{fig/plot-dce-mri.pdf}%
    \hfill
    \raisebox{-.025\height}{%
        \includegraphics[
            trim=647 270 85 277,clip,width=0.208\textwidth
        ]{fig/plot-dce-mri.pdf}
    }
    \vspace{-0.8em}
    \caption{
        DCE-MRI data of a human kidney; data courtesy of Jarle Rørvik, Haukeland
        University Hospital Bergen, Norway; taken from \cite{brehmer2018}.
        The deformation (from the left: third and fourth picture) mapping the
        template (second) to the reference (first) image, computed using our
        proposed model, is able to significantly reduce the misfit in the left
        half while fixing the spinal cord at the right edge as can be observed
        in the difference images from before (fifth) and after (last)
        registration.
    }\label{fig:result-dce-mri}
\end{figure}

\para{Image registration} %
We show that the proposed model can be applied to two-dimensional image
registration problems (Figure~\ref{fig:result-rotate} and \ref{fig:result-dce-mri}).
We used the sum of squared
distances (SSD) data term $\rho(x,z) := \frac{1}{2}\|R(x) - T(x+z)\|_2^2$ and squared Laplacian (curvature) regularization $\eta(p) := \frac{1}{2}\|\cdot\|^2$. The image values $T(x + z)$ were calculated using bilinear interpolation with Neumann boundary conditions.
After minimizing the lifted functional, we projected the solution by taking
averages over $\Gamma$ in each image pixel.

In the first experiment (Figure~\ref{fig:result-rotate}), the reference image
$R$ was synthesized by numerically rotating the template $T$ by $40$ degrees.
The grid plot of the computed deformation as well as the deformed template are visually
very close to the rigid ground-truth deformation (a rotation by 40 degrees).
Note that the method obtains almost pixel-accurate results although the range $\Gamma$ of the deformation is discretized on a disk around the origin, triangulated using only $25$ vertices, which is far less than the image resolution.

The second experiment (Figure~\ref{fig:result-dce-mri}) consists of two coronal
slices from a DCE-MRI dataset of a human kidney (data courtesy of Jarle Rørvik,
Haukeland University Hospital Bergen, Norway; taken from \cite{brehmer2018}).
The deformation computed using our proposed model is able to significantly
reduce the misfit in liver and kidney in the left half while accurately fixing
the spinal cord at the right edge.

\para{Projecting the lifted solution} \label{sec:results-proj}%
%\subsection{Projection of the lifted solution}%
%\label{sec:theory-proj}
In the scalar-valued case with first-order regularization, the minimizers of the
calibration-based lifting can be projected to minimizers of the original
problem \cite[Theorem 3.1]{pock2010}.
In our notation, the thresholding technique used there corresponds to mapping
$\mat{u}$ to
\begin{equation}\label{eq:thresholding}
    u(x) := \inf\{ t : \mat{u}_x((-\infty,t] \cap \Gamma) > s \},
\end{equation}
which is (provably) a global minimizer of the original problem
for any $s \in [0,1)$. 

\begin{figure}[t]
    \centering
    \hfill
    \includegraphics[
        trim=613 35.5 87 191,clip,width=0.29\textwidth
    ]{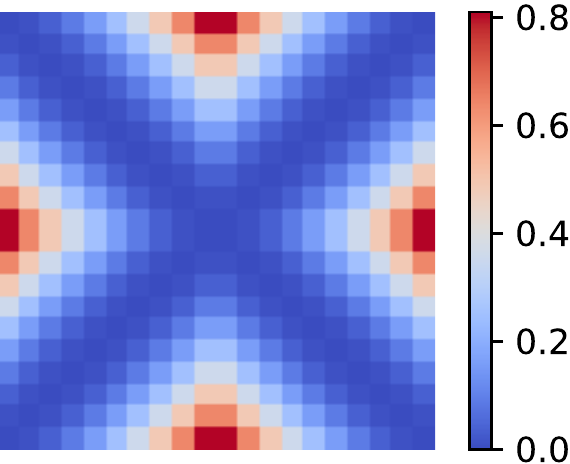}
    \hfill
    \includegraphics[
        trim=376 35.5 323 191,clip,width=0.29\textwidth
    ]{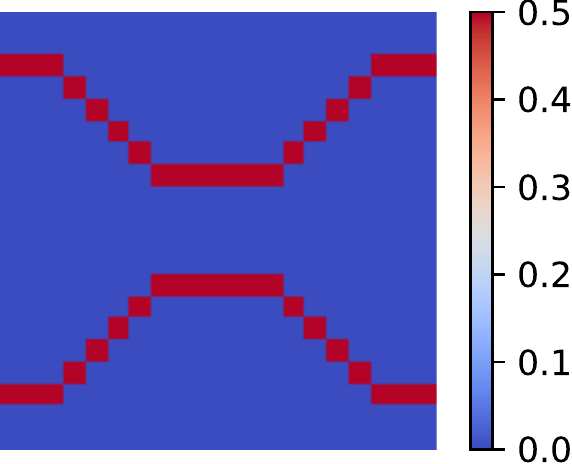}
    \hfill
    \includegraphics[
        trim=376 35.5 323 191,clip,width=0.29\textwidth
    ]{fig/plot-second-order.pdf}
    \hfill
    \vspace{-0.9em}
    \caption{Minimizers of the lifted functional for the non-convex data term $\rho(x,z) = (|x| - |z|)^2$ (left). With classical first-order total variation-regularized lifting (middle), the result is a composition of two solutions, which can be easily discriminated using thresholding. For the new second-order squared-Laplacian regularized lifting (right), this simple approach fails to separate the two possible (straight line) solutions.
    }\label{fig:result-proj}
    \vspace{-1.5em}
\end{figure}%

To investigate whether a similar property can hold in our higher-order case, we applied our model with Laplacian regularization $\eta(p) = \frac{1}{2}\|p\|^2$ as well as the
calibration method approach with total variation regularization to the
data term $\rho(x,z) = (|x| - |z|)^2$ with one-dimensional domain
$\Omega = [-1,1]$ and scalar data $\Gamma = [-1,1]$ using $20$ 
regularly-spaced discretization points (Figure~\ref{fig:result-proj}).

The result from the first-order approach is easily interpretable as a
composition of two solutions to the original problem, each of which can be
obtained by thresholding \eqref{eq:thresholding}.
In contrast, thresholding applied to the result from the second-order
approach yields the two hat functions $v_1(x) = |x|$ and $v_2(x) = -|x|$,
neither of which minimizes the original functional.
Instead, the solution turns out to be of the form
$\mat{u} = \frac{1}{2}\delta_{u_1} + \frac{1}{2}\delta_{u_2}$, where $u_1$ and
$u_2$ are in fact global minimizers of the original problem:
namely, the straight lines $u_1(x) = x$ and $u_2(x) = -x$.

\section{Conclusion}

In this work we presented a novel fully-continuous functional lifting approach for non-convex variational problems that involve Laplacian second-order terms and vectorial data, with the aim to ultimately provide sufficient optimality conditions and find global solutions despite the non-convexity. First experiments indicate that the method can produce subpixel-accurate solutions for the non-convex image registration problem. We argued that more involved projection strategies than in the classical calibration approach will be needed for obtaining a good (approximate) solution of the original problem from a solution of the lifted problem. %Further interesting research directions 
Another interesting direction for future work is the generalization to functionals
that involve arbitrary second- or higher-order terms.
%al problem,  is more involved than in the classical 
% Interesting directions for future work concern a generalization to functionals
%that involve arbitrary second-order terms, 
%  We provide
% which aims to find global solutions of non-convex  Experiments show that %We showed that the discretization of the proposed model is a generalization %of
%a continuous multi-labeling relaxation applied to the absolute Laplacian
%regularizer.
%We discussed difficulties in projecting the minimizers of the lifted problem
%to the original function space.

\para{Acknowledgments} The authors acknowledge support through DFG grant LE \mbox{4064/1-1} ``Functional Lifting 2.0: Efficient Convexifications for Imaging and Vision'' and NVIDIA Corporation.

% ==============================================================================

\bibliographystyle{splncs04}
\bibliography{paper}
% ==============================================================================

\end{document}